\documentclass[a4paper,12pt]{article}
\usepackage[cp1251]{inputenc}
\usepackage{amsmath, amsthm, amsfonts, amssymb}
\usepackage{hyperref}
\usepackage{xcolor}
\usepackage{bbm}

\makeatletter
\@addtoreset{equation}{section}
\makeatother
\makeatletter
\@addtoreset{lem}{section}
\makeatother
\makeatletter
\@addtoreset{expl}{section}
\makeatother
\makeatletter
\@addtoreset{remk}{section}
\makeatother
\makeatletter
\@addtoreset{corl}{section}
\makeatother
\makeatletter
\@addtoreset{defn}{section}
\makeatother
\newcommand{\be}{\begin{equation}}
\newcommand{\ee}{\end{equation}}
\newcommand{\bel}{\begin{equation}\label}

\newcommand{\ba}{\begin{aligned}}
\newcommand{\ea}{\end{aligned}}

\newcommand{\mul}{\begin{multline}}
\newcommand{\emul}{\begin{multline}}

\newcommand{\eee}{{\rm e}}

\DeclareMathOperator{\1}{\mathbbm{1}}

\newcommand{\ve}{\varepsilon}

\newcommand{\wt}{\widetilde}

\newcommand{\mn}{{\mathbb N}}
\newcommand{\me}{{\mathbb E}}
\newcommand{\mmp}{\mathbb{P}}
\newcommand{\mbZ}{{\mathbb Z}}
\newcommand{\mbR}{{\mathbb R}}
\newcommand{\mr}{\mathbb{R}}

\newcommand{\Pb}{\mathbb{P}}

\theoremstyle{plain}
\newtheorem{thm}{Theorem}[section]
\newtheorem{lem}{Lemma}[section]
\newtheorem{prop}{Proposition}[section]
\newtheorem{corl}{Corollary}[section]

\theoremstyle{definition}

\theoremstyle{remark}
\newtheorem{remk}{Remark}[section]

\begin{document}
\title{On a discrete approximation of a skew stable L\'{e}vy process}
\date{}
\author{Congzao Dong\footnote{School of Mathematics and Statistics, Xidian University, Xi'an, P.R. China;\newline e-mail address: czdong@xidian.edu.cn} \ \ Oleksandr Iksanov\footnote{Faculty of Computer Science and Cybernetics, Taras Shevchenko National University of Kyiv, Ukraine; e-mail address: iksan@univ.kiev.ua} \ \ Andrey Pilipenko\footnote{Institute of Mathematics of Ukrainian National Academy of Sciences, Ukraine and
Igor Sikorsky Kyiv Polytechnic Institute, Ukraine; e-mail address: apilip@imath.kiev.ua}}

\maketitle

\begin{abstract}
Iksanov and Pilipenko (2023) defined a skew stable L\'{e}vy process as a scaling limit of a sequence of perturbed at $0$ symmetric stable L\'{e}vy processes (continuous-time processes). Here, we provide a simpler construction of the skew stable L\'{e}vy process as a scaling limit of a sequence of perturbed at $0$ standard random walks (random sequences).
\end{abstract}


\hphantom{2020 Mathematics Subject Classification: } Secondary: 60G50

\vskip15pt
\section{Introduction}\label{sec:Intro}
Let $\xi_1$, $\xi_2,\ldots$ be independent copies of a random variable $\xi$ with zero mean 
and a $1$-arithmetic distribution. The latter means that the distribution of $\xi$ is concentrated on the set of integers $\mathbb{Z}$ and not concentrated on $d\mathbb{Z}$ for any  $d\geq 2$. 
Denote by $S_\xi:=(S_\xi(n))_{n\in\mn_0}$ ($\mn_0:=\mn\cup\{0\}$) the zero-delayed standard random walk with increments $\xi_n$ for $n\in\mn$, that is, $S_\xi(0):=0$ and $S_\xi(n):=\xi_1+\ldots+\xi_n$ for $n\in\mn$. Denote by $D:=D[0,\infty)$ the Skorokhod space, that is, the space of c\`{a}dl\`{a}g functions defined on $[0,\infty)$. We always assume that $D$ is endowed with the $J_1$-topology and write $\Rightarrow$ for weak convergence in this space.

It is known (see, for instance, Theorem P8 on p.~23 in \cite{Spitzer:2001}) that $S_\xi$ visits every integer point, and particularly $0$, infinitely often almost surely (a.s.). Let $X:=(X(n))_{n\in\mn_0}$ be a Markov chain with transition probabilities $$\mmp\{X(n+1)=j \ |\ X(n)=i\}=\mmp\{S_\xi(n+1)=j\ |\ S_\xi(n)=i\}$$ for integer $i\neq 0$ and some other transition probabilities for $i=0$. The chain $X$ can be thought of as a standard random walk perturbed at $0$. It is natural to ask to what extent classical functional limit theorems for standard random walks, properly scaled and normalized, have to be adjusted in the presence of a slight perturbation.

The answer is known in the case when $\sigma^2:=\me \xi^2\in (0, \infty)$ and the jumps from $0$, not necessarily identically distributed, have finite mean. Investigation of this case was initiated in the seminal article \cite{Harrison+Shepp:1981} and continued in many papers, a survey of relevant literature can be found in Section 1 of \cite{Iksanov+Pilipenko:2016}. While the weak limit on 
$D$ of $\sigma^{-1}v^{-1/2}S_\xi(\lfloor vt\rfloor)$ as $t\to\infty$ is a Brownian motion, the weak limit of $\sigma^{-1}v^{-1/2}X(\lfloor vt\rfloor)$ is a skew Brownian motion. Recall that a skew Brownian motion $(W_\gamma(t))_{t\geq 0}$ with permeability parameter $\gamma\in[-1,1]$ is a strong Markov process with $W_\gamma(0)=0$ and the transition density
$$p_t(x, y)= \varphi_t(x-y)+ \gamma\, {\rm sign}(y)\,\varphi_t(|x| + |y|),\quad t>0,\ x,y\in\mr,$$ where $\varphi_t(x) = (2\pi t)^{-1/2}\exp(-x^2/(2t))$, $t>0$, $x\in\mr$ is the density of a 
normal distribution with zero mean and variance $t$, see formula (17) in \cite{Lejay:2006}. This process behaves like a Brownian motion until hitting $0$, then its excursions ``select'' a positive or negative sign with probabilities $
(1+\gamma)/2$ and $
(1-\gamma)/2$, respectively, the subsequent evolution being analogous. It is also known (see pp.~311-312 in \cite{Harrison+Shepp:1981}) that the skew Brownian motion is a unique solution to the 
equation
\bel{eq:skewSDE}
Y(t) =W(t) +\gamma L_0^Y(t),\quad t\geq 0,
\ee
where $W$ is a Brownian motion and $L_0^Y$ is a local time of $Y$ at $0$. According to the claim on p.~312 in \cite{Harrison+Shepp:1981} there is no solution to \eqref{eq:skewSDE} if $|\gamma|>1$.

To the best of our knowledge, the situation where $\sigma^2\in (0,\infty)$ and the jumps from $0$ have infinite mean was only investigated in \cite{Pilipenko+Prykhodko:2014} and \cite{Iksanov+Pilipenko+Povar:2023} under the assumption that the jumps from $0$ are a.s.\ positive and independent with a common distribution belonging to the domain of attraction a $\beta$-stable distribution, $\beta\in(0,1)$. In the latter paper allowance is made that both $\xi$ and the jumps from $0$ are real-valued, whereas in the former these are integer-valued with $\xi$ being bounded from below by $-1$. The corresponding 
scaling limit is a Brownian motion with jump-exit from $0$ of infinite intensity, see Theorem 1.1 in \cite{Pilipenko+Prykhodko:2014} or Theorem 1.1 in \cite{Iksanov+Pilipenko+Povar:2023}.

For $\alpha\in (1,2)$, let $U_\alpha:=(U_\alpha(t))_{t\geq 0}$ be a symmetric $\alpha$-stable L\'{e}vy process with
\begin{equation}\label{eq:charU}
\me \exp({\rm i}z(U_\alpha(t)-U_\alpha(s))=\exp(-(t-s)|z|^\alpha),\quad z\in\mr,\ t>s\geq 0.
\end{equation}
One may ask how to define a skew stable L\'{e}vy process, that is, a skew version of $U_\alpha$? This intriguing problem remained open for decades. A natural definition of a skew stable L\'{e}vy process was given in the very recent paper \cite{Iksanov+Pilipenko:2023}. We stress that the approach based on selecting a sign of excursion of $U_\alpha$ does not work because any excursion of $U_\alpha$ attains positive and negative values a.s.\ 
in any neighborhood of $0$, see, for instance, Theorem 47.1 in \cite{Sato:1999}.

The idea exploited in \cite{Iksanov+Pilipenko:2023} that we briefly outline 
below is to define a skew stable L\'{e}vy process as a weak limit of certain perturbations of $U_\alpha$. Let $\zeta_1$, $\zeta_2,\ldots$ be independent copies of a random variable $\zeta$ with $\mmp\{\zeta=0\}=0$, which are also independent of $U_\alpha$. {We construct an approximating process piece-by-piece. To this end, for} each $\varepsilon>0$, with $\theta_\ve$ denoting a random variable which satisfies $\mmp\{\theta_\ve=0\}=0$ and is independent of $U_\alpha$ and $\zeta_1$, $\zeta_2,\ldots$, put
\begin{equation*}
\ba
\sigma_0:=0,\quad \sigma_{k+1}:=\inf\{t>\sigma_k\ : \ Y_\ve(\sigma_k)+ U_\alpha(t)-U_\alpha(\sigma_k)=0\}, \quad k\in\mn_0,\\
\quad Y_\ve(0)=\theta_\ve, \quad Y_\ve(\sigma_k):= \ve\zeta_k, \quad k\in\mn,\\
Y_\epsilon(t):=Y_\ve(\sigma_k)+ U_\alpha(t)-U_\alpha(\sigma_k),\quad t\in[\sigma_k, \sigma_{k+1}),\quad k\in\mn_0.
\ea
\end{equation*}
Thus, for each $\varepsilon>0$, $Y_\ve:=(Y_\ve(t))_{t\geq 0}$ makes a jump upon each arrival to $0$, and the size of the $k$th jump from $0$ is equal to $\ve\zeta_k$. The increments of $Y_\ve$ and $U_\alpha$ coincide on any time interval between successive visits of $Y_\varepsilon$ to $0$.

For a strong Markov process $X$, put
\[
\sigma(X):= \inf\{ t\geq 0\ : \ X(t)=0\}
\]
with the usual convention that the infimum taken over the empty set is equal to $+\infty$, so that $\sigma(X)$ is the first hitting time of $0$ by $X$. Denote by $R^{X}_\lambda$ and $V^{X}_\lambda$ the resolvents of the processes $X$ and $X$ killed upon hitting $0$, respectively, that is,
\[
R^{X}_\lambda f(x):=\me_x\int_0^\infty \eee^{-\lambda t} f(X(t)){\rm d}t,\quad \lambda>0
\]
and
\[
V^X_\lambda f(x):=\me_x\int_0^{\sigma(X)}\eee^{-\lambda t} f(X(t)){\rm d}t,\quad \lambda>0
\]
for bounded continuous functions $f: \mr\to\mr$. Sometimes, when there is no ambiguity, we shall write $R_\lambda$ and $V_\lambda$ 
in place of $R^X_\lambda$ and $V^X_\lambda$.

Here is a slight reformulation of Theorem A in \cite{Iksanov+Pilipenko:2023}.
\begin{prop}\label{thm:SKEW_LEVY}
(a) Let $\beta\in(0,\alpha-1)$ and $\eta^\ast$ be a measure defined by
\begin{equation}\label{eq:measure_eta}
\eta^\ast({\rm d}x)=(c_-\1_{(-\infty,\, 0)}(x)+c_+\1_{(0,\,\infty)}(x))|x|^{-(1+\beta)}{\rm d}x,\quad x\in\mathbb{R}
\end{equation}
for nonnegative $c_\pm$ satisfying $c_++c_->0$.  The function $R_\lambda$ defined by
\bel{eq:RES}
R_\lambda f(x) =V_\lambda f(x)+\frac{\int_{\mbR}V_\lambda f(y)\eta^\ast({\rm d}y)}{\lambda \int_{\mbR}V_\lambda 1 (y)\eta^\ast({\rm d}y)}\me_x \eee^{-\lambda \sigma(U_\alpha)},\quad \lambda>0,\ x\in\mr
\ee
is the resolvent of a Feller process. Here, the equality holds for any bounded continuous function $f:\mr\to\mr$. 

\noindent (b) Assume that the random variables $Y_{\ve}(0)$ converge in distribution as $\ve\to0+$ to some random variable $\theta$.

\noindent If the distribution of $\zeta$ belongs to the domain of attraction of a $\beta$-stable distribution, $\beta\in (0,\alpha-1)$,
then the processes $Y_{\ve}$ converge weakly on $D$ to a process that starts at $\theta$ and has the resolvent given in \eqref{eq:RES} with  $c_\pm$ in \eqref{eq:measure_eta} defined by
 \[
  c_\pm:=\lim_{x\to +\infty}\frac{\mmp\{\pm \zeta >x\}}{\mmp\{|\zeta|>x\}}.
 \]
If $\me |\zeta|<\infty $ or the distribution of $\zeta$ belongs to the domain of attraction of a $\beta$-stable distribution with $\beta\in (\alpha-1, 1)$, then $Y_{\ve}\Rightarrow \theta +U_\alpha$ on $D$ as $\varepsilon\to 0+$, where $\theta$ and $U_\alpha$ are independent, and $U_\alpha(0)=0$ a.s.
\end{prop}

The process $U_{\alpha,\, \beta}:=(U_{\alpha,\, \beta}(t))_{t\geq 0}$, with the resolvent given in \eqref{eq:RES}, is defined  in \cite{Iksanov+Pilipenko:2023} and called a skew $\alpha$-stable L{\'e}vy process. It is a strong Markov process that behaves like $U_\alpha$ until hitting $0$ and has a `jump-type' exit from $0$ of infinite intensity. The process $U_{\alpha,\, \beta}$ is characterized in \cite{Iksanov+Pilipenko:2023} by means of Ito's excursion theory and also as a solution to an equation involving a local time. For instance, $Y:=U_{\alpha, \beta}$ is a (weak) solution to
\[
Y(t) =Y(0)+ U_\alpha(t) +\mathcal{S}_\beta(L^Y_0(t)),\quad t\geq 0,
\]
where $\mathcal{S}_\beta$ is a $\beta$-stable L{\'e}vy process, which is independent of $U_\alpha$, with the L{\'e}vy measure being a constant multiple of $\eta^\ast$ in \eqref{eq:measure_eta}, and $L^Y_0$ is the Blumenthal-Getoor local time of $Y$ at $0$.

\section{Main result} \label{section:2_3Perturbed_RW_Levy}
For a real-valued random variable $\tau$ we shall denote by $S_\tau$ a standard random walk with increments $\tau_n$ for $n\in\mn$, where $\tau_1$, $\tau_2,\ldots$ are independent copies of $\tau$. Assume that the distribution of $\tau$ belongs to the domain of attraction of a $\gamma$-stable distribution  with $\gamma\in (0,2)\backslash\{1\}$. Then
$$\mmp\{|\tau|>x\}~\sim~ x^{-\gamma}\ell(x),\quad x\to\infty$$ and $$\mmp\{\tau>x\}~\sim~c_+\mmp\{|\tau|>x\}\quad\text{and} \quad \mmp\{-\tau>x\}~\sim c_-\mmp\{|\tau|>x\},\quad x\to\infty$$ for some $\ell$ slowly varying at $\infty$ and some nonnegative $c_+$ and $c_-$ summing up to one. According to a classical Skorokhod's result (Theorem 2.7 in \cite{Skorokhod:1957})
\begin{equation}\label{eq:conv}
\Big(\frac{S_\tau(\lfloor vt\rfloor)}{c(v)}\Big)_{t\geq 0}~\Rightarrow~ \mathcal{S}_\gamma,\quad v\to\infty
\end{equation}
on $D$, where $c$ is a positive function satisfying $\lim_{x\to\infty}x\mmp\{|\tau|>c(x)\}=1$ and $\mathcal{S}_\gamma:=(\mathcal{S}_\gamma(t))_{t\geq 0}$ is a $\gamma$-stable L{\'e}vy process with the characteristic function
\begin{equation}\label{chf}
\me \exp({\rm i}z\mathcal{S}_\gamma(t))= \exp(t|z|^\gamma(\Gamma(2-\gamma)/(\gamma-1))(\cos(\pi\gamma/2)-{\rm i}(c_+-c_-) \sin(\pi\gamma/2)\,{\rm sign}\,z)
\end{equation}
for $z\in\mr$. Here, $\Gamma$ denotes the gamma function. If, for instance, $\mmp\{|\tau|>x\}\sim Ax^{-\gamma}$ as $x\to\infty$ for a constant $A\in (0,\infty)$, then one may take $c(v)=(Av)^{1/\gamma}$. In general, $c$ is a function which is regularly varying at $\infty$ of index $1/\gamma$.

As in Section \ref{sec:Intro}, let $(X(n))_{n\in\mn_0}$ be a standard random walk perturbed at $0$, that is, a Markov chain with transition probabilities
\[
 \Pb\{X(n+1)=j\ |\ X(n)=i\}=
\begin{cases}
\Pb\{\xi=j-i\},~\text{if}~i\neq 0;\\
\Pb\{\eta=j\},~ \text{if}~ i= 0,\\
\end{cases}
 \]
where $\eta $ is an integer-valued random variable with $\mmp\{\eta=0\}<1$.

In addition to the conditions imposed on the distribution of $\xi$ in Section \ref{sec:Intro} we assume that the distribution of $\xi$ belongs to the domain of attraction of a symmetric $\alpha$-stable distribution with $\alpha\in (1,2)$. Thus, in the setting of the next to the last paragraph $\tau=\xi$, $\gamma=\alpha$ and $c_+=c_+=1/2$. Further,
\bel{eq:convStable1823}
\Big(\frac{S_\xi([vt])}{a(v)}\Big)_{t\geq 0}~\Rightarrow~  U_\alpha,\quad v\to\infty
\ee
on $D$, where $U_\alpha$ is a symmetric $\alpha$-stable L{\'e}vy process satisfying \eqref{eq:charU} and $U_\alpha(0)=0$ a.s., and $a$ is any positive function satisfying $$\lim_{x\to\infty}x\mmp\{|\xi|>a(x)\}=-(\Gamma(2-\alpha)/(\alpha-1))\cos(\pi\alpha/2).$$ The latter limit relation is a specialization of \eqref{eq:conv} and \eqref{chf}. Also, we assume that either $\me |\eta|<\infty$ or the distribution of $\eta$ belongs to the domain of attraction of a $\beta$-stable distribution with $\beta\in (0,1)$. In particular, in the latter case $$\Big(\frac{S_\eta(\lfloor vt\rfloor)}{c(v)}\Big)_{t\geq 0}~\Rightarrow~ \mathcal{S}_\beta,\quad v\to\infty$$ on $D$ and $\lim_{v\to\infty}(a(v)/c(v))=0$ because $\alpha>\beta$.

For each $v>0$, let $(X_v(n))_{n\in\mn_0}$ be a Markov chain having the same transition probabilities as $(X(n))_{n\in\mn_0}$ but possibly satisfying a different initial condition. We are ready to state the main result of the paper.
\begin{thm}\label{thm:lim_skew_Levy_walk}
Let $x\in\mr$ and assume that $X_v(0)/a(v)$ converges in probability to $x$ as $v\to\infty$.

\noindent  (a) If the distribution of $\eta$ belongs to the domain of attraction of a $\beta$-stable distribution with $\beta<\alpha-1$, then
 \[
\Big(\frac{X_v(\lfloor vt \rfloor)}{a(v)}\Big)_{t\geq 0}~\Rightarrow~ \big(U_{\alpha, \beta}(x,t)\big)_{t\geq 0},\quad  v\to\infty
 \]
 on $D$, where  $(U_{\alpha,\, \beta}(x,t))_{t\geq 0}$ is a skew stable L\'evy process starting from $x$.

\noindent
(b) If $\me |\eta|<\infty$ or the distribution of $\eta$ belongs to the domain of attraction of a $\beta$-stable distribution with $\beta>\alpha-1$,
then
\begin{equation}\label{eq:FCLT2}
\Big(\frac{X_v(\lfloor vt\rfloor)}{a(v)}\Big)_{t\geq 0}~\Rightarrow~ \big(x+U_\alpha(t)\big)_{t\geq 0},\quad v\to\infty,
\end{equation}
where $U_\alpha$ is a symmetric $\alpha$-stable L\'evy process satisfying $U_\alpha(0)=0$ a.s.
\end{thm}

Our proof of Theorem \ref{thm:lim_skew_Levy_walk} exploits a resolvent approach and bears significant similarity to the proof of Proposition \ref{thm:SKEW_LEVY}, which can be found  as Theorem A  in \cite{Iksanov+Pilipenko:2023}. In the cited article, the skew $\alpha$-stable L{\'e}vy process was constructed as a scaling limit of small perturbations at $0$ of a symmetric $\alpha$-stable process. The main achievement of Theorem \ref{thm:lim_skew_Levy_walk} is a new construction of a skew stable L{\'e}vy process as a scaling limit of locally perturbed standard random walks. On the technical side, a passage from continuous-time processes to random sequences requires at places additional non-trivial arguments. Last but not least, part (b) of Theorem \ref{thm:lim_skew_Levy_walk} is a discrete-time counterpart of part (b) of Proposition \ref{thm:SKEW_LEVY}. We think our proof of Theorem \ref{thm:lim_skew_Levy_walk}(b) is much simpler than the proof of Proposition \ref{thm:SKEW_LEVY}(b),  see the proof of part (b) of Theorem A  in \cite{Iksanov+Pilipenko:2023}.

There is an essential difference between the cases $\me \xi^2=\sigma^2\in (0, \infty)$ and $\sigma^2=\infty$ when the perturbations have finite means or more generally are sufficiently light-tailed. In the latter case, according to Theorem \ref{thm:lim_skew_Levy_walk}(b) the perturbations have no effect asymptotically, and the scaling limit of the locally perturbed standard random walk is, up to a shift, the same as the scaling limit of the unperturbed random walk. In the former case, according to the results discussed in the third paragraph of Section \ref{sec:Intro}, the scaling
limit of a locally perturbed standard random walk is a skew Brownian motion, rather than a Brownian motion (the scaling limit of the unperturbed random walk).

\section{Auxiliary results}\label{section4Auxiliary}
 In this section we collect several 
 results on convergence of functions and processes in the space $D$. We start by formulating a fragment of Theorem 13.2.2 on p.~430 in \cite{Whitt:2002}.
\begin{prop} \label{prop:WhittConvCompositions}
For $n\in\mn_0$, let $(f_n, g_n)\in D\times D$. Assume that, for $n\in\mn$, $g_n$ are nonnegative and nondecreasing, that $g_0$ is continuous and increasing, and that $\lim_{n\to\infty} (f_n,g_n)=(f_0,g_0)$ in the $J_1$-topology on $D\times D$. Then  $\lim_{n\to\infty} f_n\circ g_n= f_0\circ g_0$ in the $J_1$-topology on $D$, where $\circ$ denotes composition.
\end{prop}

The following fundamental result, called the Skorokhod representation theorem, allows us to treat convergence in distribution as an a.s.\ convergence. We present it as given in Theorem 3.30 on p.~56 in \cite{Kallenberg:1997}.
\begin{prop}\label{thm:SkorokhodRepresentation}
Let $(\theta_n)_{n\in\mn_0}$ be random elements in a separable metric space and assume that $\theta_n$ converges in distribution to $\theta_0$ as $n\to\infty$. Then there is a probability space and a sequence $(\tilde \theta_n)_{n\in\mn_0}$ defined on this space such that, for each $n\in\mn_0$, $\tilde \theta_n$ has the same distribution as $\theta_n$ and $$\lim_{n\to\infty} \tilde \theta_n= \tilde \theta_0\quad\text{{\rm a.s.}}$$
\end{prop}
\begin{remk}\label{remk:corl_Skor_repr}
Let $((f_n, g_n))_{n\in\mn_0}$ be a sequence of stochastic processes in $D\times D$, whose paths a.s. satisfy the assumptions of Proposition \ref{prop:WhittConvCompositions}.
An appeal to Proposition \ref{thm:SkorokhodRepresentation} enables us to deduce the weak convergence $f_n\circ g_n\Rightarrow f_0\circ g_0$ on $D$ as $n\to\infty$.

Proposition \ref{thm:SkorokhodRepresentation} is applicable both in the aforementioned setting and in the other parts of the paper because all the function spaces appearing in the text (the spaces of continuous functions, monotone functions, bounded c\`{a}dl\`{a}g functions) are measurable subsets of $D$, see, for instance, p.~429 in \cite{Whitt:2002}.
\end{remk}

Let $(X(t))_{t\geq 0}$ be a time-homogeneous Markov process on $\mr$ with a family of transition probabilities $$P(t,x,A)= \mmp\{X(t)\in A \ | \ X(0)=x\}$$ for $t\geq 0$, $x\in \mr$ and Borel sets $A$ on $\mr$. Denote by $(P_t)_{t\geq 0}$ and $R_\lambda$ the semigroup and the resolvent of $X$ defined by
\[
P_tf(x)=\me_x f(X(t))=\int_\mr f(y)P(t,x,{\rm d}y),\quad t\geq 0
\]
and
\[
R_\lambda f(x)=\int_0^\infty \eee^{-\lambda t} P_tf(x){\rm d}t=\int_0^\infty \eee^{-\lambda t}\me _xf(X(t)){\rm d}t,\quad x\in\mr, \lambda>0
\]
for bounded continuous functions $f:\mr\to\mr$.

Let $C_0(\mr)$ be the Banach space of continuous functions on $\mr$ vanishing at $\pm\infty$ equipped with the supremum norm $\|f\|=\sup_{x\in\mr}|f(x)|$.
Recall that a Feller process is a strong Markov process which has a strongly continuous semigroup on $C_0(\mr)$ and possesses a c\`{a}dl\`{a}g modification. In the sequel, we tacitly assume that the paths of a Feller process itself are c\`{a}dl\`{a}g.

In the proof of Theorem \ref{thm:lim_skew_Levy_walk} we intend to approximate a Feller process taking values in $\mr$ by a sequence of continuous time Markov chains taking values in subsets of $\mr$. Furthermore, the subsets are different for different elements of the sequence. As a preparation, the phase spaces of $X^{(0)}$, $X^{(1)},\ldots$ in the following result are allowed to be different.
\begin{prop}\label{thm:weak_convergenceEK_chains}
Let $X^{(0)}$ be a Feller process on $\mr$ and, for each $n\in\mn$, $X^{(n)}$ a time-homogeneous Markov process on $G_n$, a subset of $\mr$, with paths in $D$, where $G_1$, $G_2,\ldots$ are possibly different, and transition probabilities $P^{(n)}(t,x,A)$ for $t\geq 0$, $x\in G_n$ and Borel subsets $A$ on $G_n$.
For each $n\in\mn_0$ and $\lambda>0$, denote by  $(P^{(n)}(t))_{t\geq 0}$ and $R^{(n)}_\lambda$ the semigroup and the resolvent of $X^{(n)}$.

Assume  that the random variables $X^{(n)}(0)$ converge in distribution to $X^{(0)}(0)$ as $n\to\infty$, and one of the following two conditions holds:

\noindent 1) for each $f\in C_0(\mr)$ and each $t\geq 0$,
\bel{eq:conv_semigroups_restr}
\lim_{n\to\infty}\sup_{x\in G_n}\,|P^{(n)}_t f(x)-P^{(0)}_t f(x)|=0;
\ee

\noindent 2) for each $f\in C_0(\mr)$ and each $\lambda>0$,
\bel{eq:conv_resolvents_restr}
\lim_{n\to\infty}\sup_{x\in G_n}\,|R^{(n)}_\lambda f(x)-R^{(0)}_\lambda f(x)|=0.
\ee

\noindent Then
\[
X^{(n)}~\Rightarrow~ X^{(0)},\quad n\to\infty
\]
on $D$.
\end{prop}

The proof of Proposition \ref{thm:weak_convergenceEK_chains} will be given in the Appendix.

Let
\[
\sigma_{x^*}:=\inf\{ t\geq 0\ :\ X(t)={x^*}\}
\]
be the first hitting time of $ x^*$ and $V_\lambda^X$ the resolvent of $X$ killed at $x^*$ which is defined by
\[
V_\lambda f(x)= V_\lambda^X f(x):=\me_x \int_0^{\sigma_{x^*}}\eee^{-\lambda t} f(X(t)) {\rm d}t,\quad x\in\mr.
\]
The next result provides a useful representation of the resolvent of $X$.
\begin{lem}\label{lem:4_resolvent_repr_sigma}
For any strong Markov process $X$,
\bel{eq:resolvent_general}
R_\lambda f(x)=  V_\lambda f(x)+\me_x \eee^{-\lambda \sigma_{x^*}} R_\lambda f(x^*),\quad x\in\mr.
\ee
\end{lem}
\begin{proof}
 This is a standard fact, see, for instance, formula (1.2) on p.~133 in \cite{Blumenthal:1992}.
\end{proof}

Let $a>0$ be a fixed parameter, $\tau$ a random variable with $\mmp\{\tau=0\}=0$, and $X$ a Feller process that visits $0$ a.s.\ for any starting point $x\in\mr$. Construct a {\it holding and jumping process} $X_{\tau,\, a}$ as follows. The process starts at $x$ and behaves like $X$ until the first visit to $0$. Then it spends at $0$ a random period of time having an exponential distribution of mean $1/a$. Afterwards, it makes a jump, whose size has the same distribution as $\tau$, and then behaves like $X$ until the next visit to $0$. The evolution just described then iterates, and all the excursions are independent (an excursion is a path between two successive visits to $0$). The so constructed process $X_{\tau,\, a}$ is strong Markov.

Given next is the result that can be found in formula (2.2) on p.~137 in \cite{Blumenthal:1992}.
\begin{lem}\label{lem:jumphold}
Let $R_\lambda^{\tau,\, a}$ be the resolvent  of holding and jumping process $X_{\tau,\, a}$. Then, for $\lambda>0$,
\bel{eq:4ResplventHoldJump}
\lambda R_\lambda^{\tau,\, a} f(0)=
\frac{a^{-1} f(0)+\me [(V_\lambda f) (\tau)]}{a^{-1}+\lambda^{-1}\me_\tau(1-\eee^{-\lambda \sigma})}=
\frac{a^{-1}f(0)+\me [(V_\lambda f) (\tau)]}{a^{-1}+\me [(V_\lambda 1) (\tau)]},
\ee
where $V_\lambda$ is the resolvent of $X_{\tau,\, a}$ killed at $0$ and $\sigma$ is the first hitting time of $0$ by $X_{\tau,\, a}$.
\end{lem}

Observe that the resolvent of $X_{\tau,\, a}$ can be calculated with the help of Lemma \ref{lem:4_resolvent_repr_sigma} (with $x^\ast=0$) and Lemma \ref{lem:jumphold}.

\section{Proofs}

\subsection{Proof of Theorem \ref{thm:lim_skew_Levy_walk}(a)}
We find it useful to Poissonize, for each $v>0$, the process $X_v:=(X_v(t))_{t\geq 0}$ defined by $X_v(t):=X_v(\lfloor vt\rfloor)/a(v)$ for $t\geq 0$. To this end, let $(N(t))_{t\geq 0}$ denote a Poisson process on $[0,\infty)$ of unit intensity, which is independent of $((X_v(k))_{k\in\mn_0})_{v>0}$. For each $v>0$, define now $\wt X_v:=(\wt X_v(t))_{t\geq 0}$, a Poissonized version of $X_v$, by $$\wt X_v(t):=\frac{ X_v(N(vt))}{a(v)},\quad t\geq 0.$$ The Poissonized version $\wt X_v$ is a continuous-time Markov chain. The sizes of its jumps are the same as those of $X_v$, but unlike in $X_v$ the jumps occur at random epochs given by the successive positions of a standard random walk with exponentially distributed increments of mean $1/v$. The process $\wt X_v$ is an instance of the holding and jumping process discussed in the paragraph preceding Lemma \ref{lem:jumphold}. The main reason behind using the Poissonization in the present setting is availability of formula \eqref{eq:4ResplventHoldJump}.

For each $T>0$,
\begin{equation}\label{eq:pois}
\lim_{v\to\infty}\sup_{t\in [0,\,T]}|v^{-1}N(vt)-t|=0\quad \text{a.s.}
\end{equation}
Since the limit function is non-random, continuous and increasing, 
Proposition \ref{prop:WhittConvCompositions} and Remark \ref{remk:corl_Skor_repr} tell us that the weak limits of $X_v$ and $\wt X_v$ are the same, provided these exist. In particular, it is enough to prove that
\begin{equation}\label{eq:inter2}
\wt X_v~\Rightarrow~U_{\alpha,\,\beta},\quad v\to\infty
\end{equation}
on $D$. For later use, we note that, according to Proposition \ref{prop:WhittConvCompositions} and Remark \ref{remk:corl_Skor_repr}, relations \eqref{eq:convStable1823} and \eqref{eq:pois} entail
\begin{equation}\label{eq:joint}
\wt S_v~\Rightarrow~ U_\alpha,\quad v\to\infty
\end{equation}
on $D$, where $\wt S_v:=(S_\xi(N(vt))/a(v))_{t\geq 0}$.

We intend to prove \eqref{eq:inter2} with the help of Proposition \ref{thm:weak_convergenceEK_chains}. Since $U_{\alpha,\,\beta}$ and $\wt X_v$, $v>0$ are strong Markov processes, invoking Lemma \ref{lem:4_resolvent_repr_sigma} (with $x^\ast=0$) yields, for $\lambda>0$,
$$R_\lambda^{U_{\alpha,\,\beta}} f(x)=  V_\lambda^{U_{\alpha,\,\beta}} f(x)+\me_x \eee^{-\lambda \sigma(U_{\alpha,\,\beta})} R_\lambda^{U_{\alpha,\,\beta}}f(0),\quad x\in\mr$$ and
$$R_\lambda^{\wt X_v} f(l/a(v))=  V_\lambda^{\wt X_v} f(l/a(v))+\me_{l/a(v)} \eee^{-\lambda \sigma(\wt X_v)} R_\lambda^{\wt X_v} f(0),\quad l\in\mathbb{Z}.$$
 By Proposition \ref{thm:weak_convergenceEK_chains}, \eqref{eq:inter2} follows if we can show that, for each $f\in C_0(\mr)$ and $\lambda>0$,
\bel{eq:23_107}
\lim_{v\to\infty} \sup_{l\in\mbZ} | V^{\wt X_v}_{\lambda}f(l/a(v))-V^{U_{\alpha,\,\beta}}_\lambda
f(l/a(v))|=0,
\ee
\begin{equation}
\lim_{v\to\infty} \sup_{l\in\mbZ}\,\big|\me_{l/a(v)}\eee^{-\lambda \sigma ({\wt X_v}) }-\me_{l/a(v)} \eee^{-\lambda \sigma (U_{\alpha,\, \beta})}\big|=0 \label{eq:23_111}
\end{equation}
and
\bel{eq:conv_res_0_skewLevy_seq}
\lim_{v\to\infty}  \big| R^{\wt X_v}_{\lambda }f( 0)-R^{U_{\alpha,\beta}}_\lambda f(0)\big|=0.
\ee

Observe that, for each $v>0$, the conditional distribution of $(X_v(k)\1_{\{k\leq \sigma(X_v)\}})_{k\in\mn_0}$ given $X_v(0)=x$ is the same as the conditional distribution of $(S_\xi(k)\1_{\{k\leq \sigma(S_\xi)\}})_{k\in\mn_0}$ given $S_\xi(0)=x$. This implies that
\begin{equation}\label{eq:resol}
V^{\wt X_v}_\lambda= V^{\wt S_v}_\lambda,\quad \lambda>0
\end{equation}
and, for each $l\in\mn$, $\me_{l/a(v)}\eee^{-\lambda \sigma(\wt X_v)}=\me_{l/a(v)}\eee^{-\lambda\sigma(\wt S_v)}=\me_l \eee^{-(\lambda/v)\sigma(S_\xi\circ N)}$, $\lambda\geq 0$. Here, the last equality follows by a direct computation. Also, the conditional distribution of $(U_{\alpha,\,\beta}(t)\1_{\{t<\sigma(U_{\alpha,\,\beta})\}})_{t\geq 0}$ given $U_{\alpha,\,\beta}(0)=x$ is the same as the conditional distribution of $(U_\alpha (t)\1_{\{t<\sigma(U_\alpha)\}})_{t\geq 0}$ given $U_\alpha(0)=x$. This entails $V^{U_{\alpha,\,\beta}}_\lambda= V^{U_\alpha}_\lambda$, $\lambda>0$ and, for each $x\in\mr$, $\me_x\eee^{-\lambda\sigma(U_{\alpha,\,\beta})}=\me_x\eee^{-\lambda\sigma(U_\alpha)}$, $\lambda\geq 0$. As a consequence,
\eqref{eq:23_107} and \eqref{eq:23_111} are equivalent to
\bel{eq:23_1071}
\lim_{v\to\infty} \sup_{l\in\mbZ} | V^{\wt S_v}_{\lambda}f(l/a(v))
-V^{U_{\alpha }}_\lambda
f(l/a(v))|=0
\ee
and
\begin{multline}\label{eq:conv_hitting_times_Levy_walks}
\lim_{v\to\infty} \sup_{l\in\mbZ} \big|\me_{l/a(v)} \eee^{-\lambda \sigma (\wt S_v)} -\me_{l/a(v)}\eee^{-\lambda \sigma (U_\alpha)}\big|\\=\lim_{v\to\infty} \sup_{l\in\mbZ}\, \big|\me_l \eee^{-(\lambda/v) \sigma (S_\xi\circ N)}- \me_{l/a(v)}  \eee^{-\lambda \sigma (U_\alpha)}\big|=0.
\end{multline}

Another application of Lemma \ref{lem:4_resolvent_repr_sigma} (with $x^\ast=0$) to strong Markov processes $U_\alpha$ and $\wt S_v$, $v>0$ enables us to conclude that, for $\lambda>0$,
$$R_\lambda^{U_\alpha} f(x)=  V_\lambda^{U_\alpha} f(x)+\me_x \eee^{-\lambda \sigma(U_\alpha)} R_\lambda^{U_\alpha}f(0),\quad x\in\mr$$ and
$$R_\lambda^{\wt S_v} f(l/a(v))=  V_\lambda^{\wt S_v} f(l/a(v))+\me_{l/a(v)} \eee^{-\lambda \sigma(\wt S_v)} R_\lambda^{\wt S_v} f(0),\quad l\in\mathbb{Z}.$$ Thus, if we can prove \eqref{eq:conv_hitting_times_Levy_walks}
and
\begin{multline}\label{eq:conv_resolvents_stable_walks}
\lim_{v\to\infty}  \sup_{x\in \mbR}\, |R_\lambda^{\wt S_v} f(x )- R_\lambda^{U_\alpha}f(x)|=\\
\lim_{v\to\infty}  \sup_{x\in \mbR}\,\Big|\int_0^\infty
\Big(\me f(x+ S_\xi(N(vt))/a(v))- \me f(x+U_\alpha(t) )\Big) \eee^{-\lambda t}{\rm d}t\Big|=0
\end{multline}
for $f\in C_0(\mbR)$, then \eqref{eq:23_1071} holds. Once this is done, the only remaining thing is to check \eqref{eq:conv_res_0_skewLevy_seq}.

\noindent {\sc Proof of \eqref{eq:conv_resolvents_stable_walks}}. Note that each $f\in C_0(\mr)$ is uniformly continuous and put, for $\gamma>0$, $\omega_f(\gamma):=\sup_{x,y\in\mr,\, |x-y|\leq \gamma}\,|f(x)-f(y)|$. Let $(v_k)_{k\in\mn}$ be any sequence of positive numbers satisfying $\lim_{k\to\infty} v_k=\infty$. Using \eqref{eq:joint} together with the Skorokhod representation theorem (Proposition \ref{thm:SkorokhodRepresentation}) we conclude that there exist $(\hat S_{v_k})_{k\in\mn}$, versions of $(\wt S_{v_k})_{k\in\mn}$, and $\hat U_\alpha$, a version of $U_\alpha$,
such that
\[
\lim_{k\to\infty}\hat S_{v_k}(t)=\hat U_\alpha(t)\quad\text{a.s.}
\]
on $D.$ In particular, this entails the a.s.\ convergence for almost all $t\geq 0$ with respect to Lebesgue measure.
Hence,
\begin{multline*}
\lim_{k\to\infty}  \sup_{x\in \mbR}\,\int_0^\infty
\big|\me f(x+S_\xi(N(v_kt))/a(v_k))- \me f(x+U_\alpha(t) )\big| \eee^{-\lambda t}{\rm d}t\\=
\lim_{k\to\infty}  \sup_{x\in \mbR}\,\int_0^\infty
 \me \big|f(x+ \hat S_{v_k}(t))- f(x+\hat U_\alpha(t))\big| \eee^{-\lambda t} {\rm d}t\\\leq
\lim_{k\to\infty} \int_0^\infty \me \big[\big(\omega_f\big(\hat S_{v_k}(t)- \hat U_\alpha(t)\big)\big)\wedge (2\|f\|)\big]  \eee^{-\lambda t}{\rm d}t=0,
\end{multline*}
where the last equality is justified by the Lebesgue dominated convergence theorem. Since the diverging sequence $(v_k)_{k\in\mn}$ is arbitrary, 
the proof of \eqref{eq:conv_resolvents_stable_walks} is complete.

It follows from Corollary 18 on p.~64 in \cite{Bertoin:1996} that
\bel{eq:Laplica_hitting_Levy}
\me_x \eee^{-\lambda \sigma(U_\alpha)}= \frac{v_\lambda(-x)}{v_\lambda(0)},\quad x\in\mathbb{R}, \ \lambda>0,
\ee
where $v_\lambda(x,y)=v_\lambda(y-x)$, $x,y\in\mathbb{R}$  is the density of the resolvent kernel of $U_\alpha$. It is known that
\bel{eq:resolvent_stable}
v_\lambda(x)=\frac{1}{ \pi} \int_0^\infty\frac{\cos(x\theta)}{\lambda+\theta^\alpha} {\rm d}\theta,\quad x\in\mathbb{R}.
\ee
According to the last cited result, formula \eqref{eq:Laplica_hitting_Levy} is valid for any L{\'e}vy process, whose resolvent kernel is absolutely continuous with a bounded density.

To prove \eqref{eq:conv_hitting_times_Levy_walks} we first derive in Corollary \ref{corl:formula_generating_hitting} a formula for $\me_l \eee^{-\lambda \sigma (S_\xi\circ N)}$. As a preparation, we start with an auxiliary result.
\begin{lem}\label{lem:formula_generating_hitting}
Let $Y:=(Y(k))_{k\in\mn_0}$ be a Markov chain on a finite or countable set $G$. For $x^\ast\in G$, put $\sigma:=\sigma_{x^*}:=\inf\{k\in\mn_0\ :\ Y(k)=x^*\}$.
Then
\bel{eq:formula_hitting_generating}
\me_x s^\sigma=\frac{u_s(x,x^*)}{u_s(x^*,x^*)},\quad x\in G, \ |s|<1,
\ee
where $u_s(x,x^*)=\sum_{k\geq 0} s^k \Pb\{Y(k)=x^*\ | \ Y(0)=x\}$ for $x\in G$.

\noindent In particular, if $x^*=0$, $Y(k)=S_\tau (k)$ for $k\in\mn_0$ and $(Y(k))_{k\in\mn_0}$ lives on the lattice $G=a\mbZ$ for some $a>0$, then
\bel{eq:generating_hitting}
\me_x s^\sigma =\frac{u_s(-x)}{u_s(0)},\quad x\in a\mathbb{Z}, \ |s|<1,
\ee
where
\[
u_s(x)=u_s(x,0)= \sum_{k\geq 0}s^k \mmp\{S_\tau(k)=0\ | \ S_\tau(0)=x\}=\1_{\{0\}}(x)+\sum_{k\geq 1} s^k \Pb\{S_\tau (k)=-x\}.
\]
Alternatively,
\begin{equation}\label{eq:formula}
u_s(-x)=\frac{1}{2\pi}\int_{-\pi}^\pi \frac{\eee^{{\rm i}x\theta}}{1-s \me \eee^{{\rm i}\theta \tau}}{\rm d}\theta,\quad x\in a\mbZ, \ |s|<1.
\end{equation}
\end{lem}
\begin{proof}
Denote by $R_s$ the resolvent of $Y$, so that $$R_s f (x)= \sum_{k\geq 0} s^k \me_x f(Y(k)),\quad x\in G, \ |s|<1$$ for bounded measurable functions $f:G\to\mr$. The $R_s$ satisfies a formula similar to \eqref{eq:resolvent_general}
\[
R_s f(x) =\sum_{k=0}^{\sigma-1} s^k \me_x f(Y(k))+ \me_x s^\sigma R_s f(x^\ast),\quad x\in G, \ |s|<1.
\]
Put $f(x)=\1_{\{x^\ast\}}(x)$. Then $R_s f(x)=u_s(x,x^\ast)$, $x\in G$, $|s|<1$, whereas the first summand on the right-hand side vanishes. This proves \eqref{eq:formula_hitting_generating}. Formula \eqref{eq:generating_hitting} is just a specialization of \eqref{eq:formula_hitting_generating}. To prove \eqref{eq:formula}, write with the help of Fubini's theorem
\begin{multline*}
\sum_{x\in a\mbZ} u_s(-x)\eee^{{\rm i}x\theta}=1+\sum_{x\in a\mbZ} \sum_{k\geq 1} s^k \Pb\{S_\tau(k)=x\}\eee^{{\rm i}x\theta}=1+\sum_{k\geq 1}s^k
\sum_{x\in a\mbZ}\Pb\{S_\tau(k)=x\}\eee^{{\rm i}x\theta}\\= 1+\sum_{k\geq 1}s^k (\me \eee^{{\rm i}\theta \tau})^k
= \frac{1}{1-s \me \eee^{{\rm i}\theta \tau}},\quad |s|<1,\ \theta\in\mbR.
\end{multline*}
With this at hand, \eqref{eq:formula} is an immediate consequence of a standard inversion formula.
\end{proof}
We stress that a continuous-time formula \eqref{eq:Laplica_hitting_Levy} rests on non-trivial potential-analytic results, whereas a discrete-time formula \eqref{eq:formula_hitting_generating} is rather simple.
\begin{corl}\label{corl:formula_generating_hitting}
Let the assumptions and notation of Lemma \ref{lem:formula_generating_hitting} be in force. Denote by $(N_\rho(t))_{t\geq 0}$ a Poisson process on $[0,\infty)$ of intensity $\rho>0$, which is independent of $Y$, and put $\wt Y(t)= Y(N_\rho(t))$ for $t\geq 0$. For $x^\ast\in G$, put $\wt \sigma:=\wt \sigma_{x^\ast}:=\inf\{t\geq 0\ :\ \wt Y(t)=x^\ast\}.$
Then
\bel{eq:formula_hitting_generating1}
\me_x \eee^{-\lambda\wt\sigma} =\frac{\hat u_\lambda(x, x^\ast)}{\hat u_\lambda(x^\ast,x^\ast)}=\frac{u_s(x,x^*)}{u_s(x^*,x^*)},\quad x\in G, \ \lambda>0,
\ee
where $s=\rho/(\lambda+\rho)$ and $\hat u_\lambda(x,x^\ast)=\int_0^\infty \eee^{-\lambda t} \Pb\{\wt Y(t)=x^\ast\ | \ \wt Y(0)=x\}{\rm d}t$.

\noindent In particular, if $x^*=0$, $Y(k)=S_\tau(k)$ for $k\in\mn_0$ and $(Y(k))_{k\in\mn_0}$ lives on the lattice $G=a\mbZ$ for some $a>0$,
then
\bel{eq:generating_hitting1}
\me_x \eee^{-\lambda\wt\sigma} =\frac{\hat u_\lambda(-x)}{\hat u_\lambda (0)}=\frac{u_s(-x)}{u_s (0)},\quad x\in a\mbZ, \ \lambda>0,
\ee
where $s=\rho/(\lambda+\rho)$ and
\begin{multline*}
\hat u_\lambda(x)=\hat u_\lambda(x,0)= \int_0^\infty \eee^{-\lambda t}\Pb\{\wt Y(t)=0\ | \ \wt Y(0)=x){\rm d}t\\=\int_0^\infty \eee^{-\lambda t}\Pb\{S_\tau(N_\rho(t))=-x\}{\rm d}t.
\end{multline*}
\end{corl}
\begin{proof}
The first equality in \eqref{eq:formula_hitting_generating1} follows from \eqref{eq:resolvent_general} and the argument used for the proof of \eqref{eq:formula_hitting_generating}. To prove the second equality in \eqref{eq:formula_hitting_generating1} we shall derive a formula relating the resolvent $R^{\wt Y}_\lambda$ of $\wt Y$ to the resolvent $R_s$ of $Y$. By a repeated application of Fubini's theorem
\begin{multline*}
R^{\wt Y}_\lambda f(x)=  \me_x\int_0^\infty \eee^{-\lambda t} f(\wt Y(t)){\rm d}t=
\int_0^\infty \eee^{-\lambda t} \sum_{k\geq 0} \eee^{-\rho t}\frac{(\rho t)^k}{k!}\me_x f(Y(k)){\rm d}t \\= \sum_{k\geq 0} \frac{\rho^k}{k!}\me_x f(Y(k)) \int_0^\infty t^k \eee^{-(\lambda+\rho) t} {\rm d}t= \frac{1}{\lambda+\rho} \sum_{k\geq 0} \Big(\frac{ \rho }{\lambda+\rho}\Big)^k \me_x f(Y(k))=\frac{ 1 }{\lambda+\rho}R_s f(x),
 \end{multline*}
where $s=\rho/(\lambda+\rho)$. Putting $f(x)=\1_{\{x^\ast\}}(x)$ for $x\in G$, we infer $$\hat u_\lambda(x,x^\ast)=(\lambda+\rho)^{-1}u_s(x,x^\ast)$$ for $x\in G$ and the same $s$ as before, thereby justifying the second equality in \eqref{eq:formula_hitting_generating1}. Formula \eqref{eq:generating_hitting1} is a specialization of \eqref{eq:formula_hitting_generating1}.
\end{proof}

\noindent {\sc Proof of \eqref{eq:conv_hitting_times_Levy_walks}}. We shall prove \eqref{eq:conv_hitting_times_Levy_walks} in an equivalent form:
\begin{equation}\label{eq:410}
\lim_{v\to\infty}\sup_{x\in\mr}\,\Big|\me_{\lfloor xa(v)\rfloor} \eee^{-(\lambda/v)\sigma(S_\xi\circ N)}-\me_{\lfloor xa(v)\rfloor/a(v)}\eee^{-\lambda\sigma(U_\alpha)}\Big|=0.
\end{equation}
We shall use the following representation: for $x\in\mr$,
\begin{multline}\label{eq:res_walk_hitting_laplace378}
\me_{\lfloor x a(v)\rfloor} \eee^{-(\lambda/v) \sigma(S_\xi\circ N)}= u_s(-\lfloor x a(v)\rfloor)/ u_s(0)=
\int_{-\pi}^\pi \frac{\eee^{{\rm i}\theta \lfloor x a(v)\rfloor}}{1-s \psi(\theta)}{\rm d}\theta\Big/
\int_{-\pi}^\pi \frac{{\rm d}\theta}{1-s \psi(\theta)}\\=\int_{-\pi}^\pi \frac{\eee^{{\rm i}\theta \lfloor x a(v)\rfloor}}{\lambda+v(1-\psi(\theta))}{\rm d}\theta\Big/\int_{-\pi}^\pi \frac{{\rm d}\theta}{\lambda+v(1-\psi(\theta))}\\=\int_{-\pi a(v)}^{\pi a(v)} \frac{\eee^{{\rm i}\theta \lfloor x a(v)\rfloor/a(v)}}{\lambda+v(1-\psi(\theta/a(v)))}{\rm d}\theta\Big/\int_{-\pi a(v)}^{\pi a(v)} \frac{{\rm d}\theta}{\lambda+v(1-\psi(\theta/a(v)))}
\end{multline}
with 
$s=v/(v+\lambda) 
$ and $\psi(\theta)=\me \eee^{{\rm i}\theta\xi}$ for $\theta\in\mr$, where the first equality is a specialization of \eqref{eq:generating_hitting1} for $Y=S_\xi$, $\rho=1$ and $a=1$, and the second equality follows from \eqref{eq:formula} with $\tau=\xi$. Further, by \eqref{eq:Laplica_hitting_Levy} and \eqref{eq:resolvent_stable},
\begin{multline*}
\me_{\lfloor xa(v)\rfloor/a(v)}\eee^{-\lambda\sigma(U_\alpha)}=\int_0^\infty\frac{\cos(\theta\lfloor xa(v)\rfloor/a(v))}{\lambda+\theta^\alpha} {\rm d}\theta\Big/\int_0^\infty\frac{{\rm d}\theta}{\lambda+\theta^\alpha}\\=\int_\mr \frac{\eee^{{\rm i}\theta\lfloor xa(v)\rfloor/a(v)}}{\lambda+|\theta|^\alpha} {\rm d}\theta\Big/\int_\mr \frac{{\rm d}\theta}{\lambda+|\theta|^\alpha}.
\end{multline*}
Summarizing, \eqref{eq:410} is a consequence of
\begin{equation}\label{eq:inter234}
\lim_{v\to\infty}\sup_{x\in (a(v))^{-1}\mathbb{Z}}\,\Big|\int_{-\pi a(v)}^{\pi a(v)} \frac{\eee^{{\rm i}\theta \lfloor x a(v)\rfloor/a(v)}}{\lambda+v(1-\psi(\theta/a(v)))}{\rm d}\theta-\int_\mr \frac{\eee^{{\rm i}\theta\lfloor xa(v)\rfloor/a(v)}}{\lambda+|\theta|^\alpha} {\rm d}\theta\Big|=0.
\end{equation}
To prove this, write, for any $A>1$, some $\varepsilon\in (0,\pi)$ to be specified later and large enough $v$,
\begin{multline*}
\Big|\int_{-\pi a(v)}^{\pi a(v)} \frac{\eee^{{\rm i}\theta \lfloor x a(v)\rfloor/a(v)}}{\lambda+v(1-\psi(\theta/a(v)))}{\rm d}\theta-\int_\mr \frac{\eee^{{\rm i}\theta\lfloor xa(v)\rfloor/a(v)}}{\lambda+|\theta|^\alpha} {\rm d}\theta\Big|\\\leq \int_{-A}^A \Big|\frac{1}{\lambda+v(1-\psi(\theta/a(v)))}-\frac{1}{\lambda+|\theta|^\alpha}\Big|{\rm d}\theta+\int_{A\leq |\theta|\leq \varepsilon a(v)}\frac{{\rm d}\theta}{|\lambda+v(1-\psi(\theta/a(v)))|}\\+\int_{|\theta|>A} \frac{{\rm d}\theta}{\lambda+|\theta|^\alpha}+ \int_{\ve a(v)\leq |\theta|\leq \pi a(v)}\frac{{\rm d}\theta}{|\lambda+v(1-\psi(\theta/a(v)))|}=:I(v,A)+J(v,A)+ K(A)+M(v).
\end{multline*}

A specialization of \eqref{eq:convStable1823} to a one-dimensional convergence entails
\begin{equation}\label{eq:char}
\lim_{v\to\infty} v(1- \psi(\theta/a(v)))=|\theta|^\alpha
\end{equation}
locally uniformly in $\theta$, whence $\lim_{v\to\infty} I(v,A)=0$.

Relation \eqref{eq:char} entails $$\lim_{v\to\infty}\frac{|1- \psi(\theta/a(v))|}{|1-\psi(1/a(v))|}=|\theta|^\alpha,$$ which shows that the functions $\theta\mapsto|1-\psi(\theta)|$, $\theta>0$ and $\theta\mapsto |1-\psi(-\theta)|$, $\theta>0$ are regularly varying at $0$ of index $\alpha$. By an analogue of Potter's bound (Theorem 1.5.6 in \cite{BGT:1989}), given $c\in (0,1)$ and $\delta\in (0,\alpha-1)$ there exists $\varepsilon>0$ such that
\begin{equation}\label{eq:impo}
v|1- \psi(\theta/a(v))| \geq c (|\theta|^{\alpha+\delta}\wedge |\theta|^{\alpha-\delta})
\end{equation}
for all $\theta\in [-\ve a(v), \ve a(v)]$ and large $v$. Hence, $$J(v, A)\leq \int_{|\theta|\geq A}\frac{{\rm d}\theta}{c|\theta|^{\alpha-\delta}}~\to~ 0,\quad A\to\infty.$$ Also, trivially, $$\lim_{A\to\infty}K(A)=0.$$

Since the distribution of $\xi$ is $1$-arithmetic by assumption we conclude that $\psi(\theta)=1$ if, and only if, $\theta=2\pi n$, $n\in \mbZ$. In particular, $\min_{\ve\leq |\theta|\leq\pi}|1- \psi(\theta)|>0$. Thus, $$M(v)\leq \frac{a(v)}{v}    \frac{2(\pi-\ve)
}{\min_{\ve\leq |\theta|\leq\pi} |1- \psi(\theta)|} \to 0,\quad v\to\infty$$ because $a$ is regularly varying at $\infty$ of index $1/\alpha<1$.

Combining fragments together we arrive at \eqref{eq:inter234}, which completes the proof of \eqref{eq:conv_hitting_times_Levy_walks}.

\noindent {\sc Proof of \eqref{eq:conv_res_0_skewLevy_seq}}. It follows from \eqref{eq:4ResplventHoldJump} with $\tau=\eta/a(v)$ that, for $\lambda>0$,
\begin{equation}\label{eq:resol_walks}
\lambda R_\lambda^{\wt X_v} f(0)=\frac{f(0)/v+\me \big[\big(V^{\wt X_v}_\lambda f\big) (\eta/a(v))\big]}
{1/v+\me \big[\big(V^{\wt X_v}_\lambda 1 \big) (\eta/a(v))\big]}
=\frac{f(0)/v+\me \big[\big(V^{\wt S_v}_\lambda f\big) (\eta/a(v))\big]}
{1/v+\me \big[\big(V^{\wt S_v}_\lambda 1 \big) (\eta/a(v))\big]},
\end{equation}
where the last equality is secured by \eqref{eq:resol}. Comparing a specialization of formula \eqref{eq:resolvent_general} for $U_{\alpha,\,\beta}$ and \eqref{eq:RES} we infer, for $\lambda>0$, $$\lambda R^{U_{\alpha,\,\beta}}_\lambda f(0)=\frac{\int_\mr V^{U_\alpha}_\lambda f(x)\eta^\ast({\rm d}x)}
{\int_\mr V^{U_\alpha}_\lambda 1(x)\eta^\ast({\rm d}x)},$$ where $\eta^\ast$ is a measure defined in \eqref{eq:measure_eta} with nonnegative $c_\pm$ satisfying
\begin{equation}\label{eq:c+-}
\mmp\{\pm \eta>x\}~\sim~ c_\pm \mmp\{|\eta|>x\},\quad x\to\infty
\end{equation}
(so that necessarily $c_++c_-=1$). Hence, \eqref{eq:conv_res_0_skewLevy_seq} is equivalent to \bel{eq:conv_res_0_skewLevy_seq_limSkew}
\lim_{v\to\infty}
\frac{\me \big[\big(V^{\wt S_v}_\lambda f\big) (\eta/a(v))\big]}
{\me \big[\big(V^{\wt S_v}_\lambda 1 \big) (\eta/a(v))\big]}
= \frac{\int_\mr V^{U_\alpha}_\lambda f(x)\eta^\ast({\rm d}x)}{\int_\mr V^{U_\alpha}_\lambda 1(x)\eta^\ast({\rm d}x)}
\ee
for $f\in C_0(\mbR)$.

Our proof of \eqref{eq:conv_res_0_skewLevy_seq_limSkew} is based on auxiliary facts to be discussed next. \begin{lem}\label{eq:reg231}
 Assume that the function $x\mapsto \Pb\{|\eta|>x\}$ is regularly varying at $+\infty$ of index $-\beta\in (-1,0)$, and relation \eqref{eq:c+-} holds.
Let $(g_u)_{u>0}$ be a family of uniformly bounded
measurable functions which satisfy the conditions:

\noindent 1) for a continuous function $g$
 \[
\lim_{u\to\infty}\sup_{x\in\mr}\,|g_u(x)-g(x)|=0;
\]

\noindent 2) for some positive constants $u_0$, $c$ and $\gamma$
\begin{equation}\label{eq:cond1001}
\sup_{u\geq u_0}\, |g_u(x)|\leq  c|x|^{\beta+\gamma},\quad x\in\mr.
\end{equation}
Then
\[
\lim_{u\to\infty}  \frac{\me g_u(\eta/ u )}{\mmp\{|\eta|>u\}}=\int_{\mathbb{R}}g(x)\eta^\ast({\rm d}x)\in \mbR.
\]
\end{lem}
\begin{remk}
By uniform boundedness of $(g_u)_{u>0}$, if inequality \eqref{eq:cond1001} holds for all $x$ in some vicinity of $0$, then it holds for all $x\in\mr$. Formulating \eqref{eq:cond1001} in the present form makes the subsequent proof notationally simpler.
\end{remk}
\begin{proof}
This result is an extension of Lemma 2.4 in \cite{Iksanov+Pilipenko:2023}. Here, we treat a family $(g_u)_{u>0}$, whereas the cited result dealt with a single function $g$, say.

Write
\[
\me g_u(\eta/u)=\int_\mbR g_u(x){\rm d}_xF(ux),\quad u>0,
\]
where $F$ is the distribution function of $\eta$. Finiteness of the expectation is secured by uniform boundedness of $(g_u)_{u>0}$. We claim it suffices to show that
\bel{eq:23_487}
\lim_{u\to\infty}\frac{ \int_\mbR h(x){\rm d}_xF(ux)}{\mmp\{|\eta|>u\}}= \int_{\mathbb{R}}h(x)\eta^\ast({\rm d}x)
\ee
for   any
bounded continuous function $h$ satisfying
\[
|h(x)|\leq  c_0|x|^{\beta+\gamma}, \ x\in\mbR
\]
for some $c_0>0$.

Note that the integral on the right-hand side of \eqref{eq:23_487} converges. Indeed, fix any $a>0$ and observe that the integral $\int_{|x|>a} h(x)\eta^\ast({\rm d}x)$ converges because $h$ is bounded, whereas convergence of $\int_{|x|\leq a}h(x)\eta^\ast({\rm d}x)$ is ensured by the last displayed inequality and $\gamma>0$.

Now we proceed to justifying the claim. Given $\delta>0$ there exists a $u_\delta$ such that
 \[
 |g_u(x)-g(x)|\leq 2 c(|x|\wedge \delta)^{\beta+\gamma},\quad x\in\mr
 \]
 whenever $u\geq u_\delta$. Hence, if \eqref{eq:23_487} holds true, then
\begin{multline*}
\limsup_{u\to\infty}\Big|\frac{\int_\mbR  g_u(x){\rm d} _x F(ux)}{\mmp\{|\eta|>u\}} -\frac{ \int_\mbR  g(x){\rm d}_xF(ux)}{\mmp\{|\eta|>u\}}\Big|\leq
\limsup_{u\to\infty} \frac{2c\int_\mbR (|x|\wedge \delta)^{\beta+\gamma}{\rm d}_xF (ux)}{\mmp\{|\eta|>u\}}\\=2c \int_{\mathbb{R}} (|x|\wedge \delta)^{\beta+\gamma}\eta^\ast({\rm d}x).
\end{multline*}
Sending $\delta\to 0+$ and invoking the Lebesgue dominated convergence theorem we conclude that the right-hand side vanishes. This justifies sufficiency of \eqref{eq:23_487}.

The remaining argument repeats verbatim the proof of Lemma 2.4 in \cite{Iksanov+Pilipenko:2023}. We omit details.
\end{proof}

\begin{lem}\label{lem:asym}
Given $\delta\in (0, (\alpha-1)\wedge (2-\alpha))$ there exist positive $v_0$ and $c_2$ such that
\[\sup_{v\geq v_0}\, \me_{\lfloor xa(v)\rfloor}\Big(1-\eee^{-(\lambda/v) \sigma (S_\xi\circ N)}\Big)\leq c_2|x|^{\alpha-1-\delta},\quad x\in\mr.\]
\end{lem}
\begin{proof}
Since the left-hand side is bounded from above by $1$ it suffices to prove the inequality for $|x|\leq 1$.

Let $\varepsilon>0$ be the same as in \eqref{eq:impo}. Using \eqref{eq:res_walk_hitting_laplace378} and changing the variable we obtain, for $x\in\mr$, 
\begin{multline*}
\me_{\lfloor xa(v)\rfloor}\Big(1- \eee^{-(\lambda/v)\sigma(S_\xi\circ N)}\Big)\\=\frac{\pi}{\varepsilon}\Big|\int_{-\varepsilon a(v)}^{\varepsilon a(v)} \frac{1-\eee^{{\rm i}\pi\theta\lfloor xa(v)\rfloor/(\varepsilon a(v))}}{\lambda+v(1-\psi(\pi \theta/(\varepsilon a(v)))}{\rm d}\theta\Big|\Big/\Big|\int_{-\pi a(v)}^{\pi a(v)} \frac{{\rm d}\theta}{\lambda+v(1-\psi(\theta/a(v)))}\Big|.
\end{multline*}
In view of \eqref{eq:inter234}
the denominator on the right-hand side converges to $\int_\mr (\lambda+|\theta|^\alpha)^{-1}{\rm d}\theta$ as $v\to\infty$. Invoking \eqref{eq:impo} in combination with $$\lim_{v\to\infty}\frac{|1-\psi(\pi\theta/(\varepsilon a(v)))|}{|1-\psi(\theta/a(v))|}=\Big(\frac{\pi}{\varepsilon}\Big)^\alpha$$ we arrive at a counterpart of \eqref{eq:impo}: given $\delta\in (\alpha-1)\wedge (2-\alpha))$, $$v|1- \psi(\pi\theta/(\varepsilon a(v)))| \geq c_1 (|\theta|^{\alpha+\delta}\wedge |\theta|^{\alpha-\delta})$$ for all $\theta\in [-\ve a(v), \ve a(v)]$, large $v$ and some finite positive constant $c_1$. Hence,
\begin{multline*}
\Big|\int_{-\varepsilon a(v)}^{\varepsilon a(v)} \frac{1-\eee^{{\rm i}\pi\theta\lfloor xa(v)\rfloor/(\varepsilon a(v))}}{\lambda+v(1-\psi(\pi \theta/(\varepsilon a(v)))}{\rm d}\theta\Big|\leq \frac{2}{c_1} \int_\mr \frac{|\sin (\pi\theta \lfloor xa(v)\rfloor/(2\varepsilon a(v)))|}{|\theta|^{\alpha+\delta}\wedge |\theta|^{\alpha-\delta}}{\rm d}\theta\\\leq \frac{2}{c_1}\Big(\frac{\pi}{2\varepsilon}|x|\int_{-1}^1 |\theta|^{1-\alpha-\delta}{\rm d}\theta+|x|^{\alpha-1-\delta}\int_\mr \frac{|\sin(\pi\theta/(2\varepsilon))|}{|\theta|^{\alpha-\delta}}{\rm d}\theta \Big)\leq c_2 |x|^{\alpha-1-\delta}
\end{multline*}
for $|x|\leq 1$.
\end{proof}

\noindent {\sc Proof of \eqref{eq:conv_res_0_skewLevy_seq_limSkew}}. For $x\in\mr$, $\lambda>0$ and bounded continuous functions $f:\mr\to\mr$, put
\begin{multline}
\hat V^{\wt S_v}_\lambda f(x):=\int_0^\infty \eee^{-\lambda t}\me \big(f(S_\xi(N(vt))/a(v))|S_\xi(N(vt))=\lfloor x a(v)\rfloor\big)\1_{\{\sigma(\wt S_v)>t\}}{\rm d}t\\
=v^{-1}\int_0^\infty \eee^{-(\lambda/v)t}\me \big(f(S_\xi(N(t))/a(v))|S_\xi(N(t))=\lfloor x a(v)\rfloor\big)\1_{\{\sigma(S_\xi\circ N)>t\}}{\rm d}t.\label{eq:aux102}
\end{multline}
Similarly, we define, for $x\in\mr$ and $\lambda>0$,
\begin{equation}\label{eq:repr}
\hat V^{\wt S_v}_\lambda 1(x):=\lambda^{-1}\me_{\lfloor x a(v)\rfloor/a(v)}(1-\eee^{-\lambda \sigma(\wt S_v)})=\lambda^{-1}\me_{\lfloor x a(v)\rfloor}(1-\eee^{-(\lambda/v) \sigma(S_\xi\circ N)}).
\end{equation}
The functions $\hat V^{\wt S_v}_\lambda f$ and $\hat V^{\wt S_v}_\lambda 1$ are piecewise constant interpolations of $V^{\wt S_v}_\lambda f$ and $V^{\wt S_v}_\lambda 1$, respectively, satisfying $\hat V^{\wt S_v}_\lambda f(x)=V^{\wt S_v}_\lambda f(x)$ and $\hat V^{\wt S_v}_\lambda 1(x)=V^{\wt S_v}_\lambda 1(x)$  for each $x\in (a(v))^{-1}\mathbb{Z}$.

Let $\lambda>0$ be fixed and $f\in C_0(\mr)$ which particularly implies that $f$ is uniformly continuous on $\mr$. We intend to apply Lemma \ref{eq:reg231} with $u=a(v)$, $g^{(1)}=V_\lambda^{U_\alpha}f$, $g^{(2)}=V_\lambda^{U_\alpha}1$, $g_u^{(1)}=\hat V^{\wt S_v}_\lambda f$ and  $g_u^{(2)}=\hat V^{\wt S_v}_\lambda 1$. It is not obvious that the so defined $g_{a(v)}^{(1)}$ is a function of $a(v)$ alone. To justify, observe that, without loss of generality, we can assume that $a$ is strictly increasing and continuous, so that the inverse function $a^{-1}$ exists. Then $v^{-1}$ and $\lambda/v$ on the right-hand side of \eqref{eq:aux102} are equal to $1/((a^{-1}\circ a)(v))$ and $\lambda/((a^{-1}\circ a)(v))$, respectively.

Now we check that the so defined functions satisfy the assumptions of Lemma \ref{eq:reg231}. (Uniform) continuity of $g^{(1)}$ is secured by boundedness and uniform continuity of $f$ in combination with the Lebesgue dominated convergence theorem. (Uniform) continuity of $g^{(2)}$ follows from
\begin{equation}\label{eq:v1}
V_\lambda^{U_\alpha}1(x)=\lambda^{-1}\me_x(1-\eee^{-\lambda \sigma(U_\alpha)}),\quad x\in\mr
\end{equation}
in combination with \eqref{eq:Laplica_hitting_Levy} and \eqref{eq:resolvent_stable}. The uniform convergence $$\lim_{u\to\infty}\sup_{x\in\mr}\,|g_u^{(1)}(x)-g^{(1)}(x)|=\lim_{v\to\infty}\sup_{x\in\mr}\,|\hat V^{\wt S_v}_\lambda f(x)-V_\lambda^{U_\alpha}f(x)|=0$$ is guaranteed by \eqref{eq:23_1071} and uniform continuity of $V_\lambda^{U_\alpha}f=g^{(1)}$. Analogously, the relation $$\lim_{u\to\infty}\sup_{x\in\mr}\,|g_u^{(2)}(x)-g^{(2)}(x)|=\lim_{v\to\infty}\sup_{x\in\mr}\,|\hat V^{\wt S_v}_\lambda 1(x)-V_\lambda^{U_\alpha}1(x)|=0$$  follows from \eqref{eq:v1}, \eqref{eq:conv_hitting_times_Levy_walks} and uniform continuity of $V_\lambda^{U_\alpha}1=g^{(2)}$. Uniform boundedness of $(g_u^{(2)})_{u>0}$ follows from representation \eqref{eq:repr} and entails uniform boundedness of $(g_u^{(1)})_{u>0}$ via
\begin{equation}\label{eq:ub}
|\hat V^{\wt S_v}_\lambda f(x)|\leq \|f\|\, \hat V^{\wt S_v}_\lambda 1(x),\quad x\in\mr
\end{equation}
for $f\in C_0(\mr)$. While the functions $g_u^{(2)}$ satisfy \eqref{eq:cond1001} with $\gamma=\alpha-1-\beta-\delta$ in view of representation \eqref{eq:repr} and Lemma \ref{lem:asym} applied for $\delta\in (0, (\alpha-1-\beta)\wedge (2-\alpha))$, the functions $g_u^{(1)}$ do so as a consequence of \eqref{eq:ub}.

Thus, all the conditions of Lemma \ref{eq:reg231} are satisfied, and an application of Lemma \ref{eq:reg231} yields \eqref{eq:conv_res_0_skewLevy_seq_limSkew} and thereupon \eqref{eq:conv_res_0_skewLevy_seq}.

The proof of Theorem \ref{thm:lim_skew_Levy_walk} (a) is complete.

\subsection{Proof of Theorem \ref{thm:lim_skew_Levy_walk}(b)}

We shall work with a particular realization of the Markov chain $X=(X(n))_{n\in\mn_0}$, still denoted by $X$ and defined by
\begin{equation} \label{tildeS}
X(n+1)=
\begin{cases}
X(n)+\xi_{n+1-T(n)}, & \text{if}~~ X(n)\neq 0, \\
X(n) + \eta_{T(n)}, & \text{if}~~ X(n)=0
\end{cases}
\end{equation}
for $n\in\mn_0$, where $T(n):=\sum_{k=0}^n\1_{\{X(k)=0\}}$ and $\eta_1$, $\eta_2,\ldots$ are independent copies of $\eta$, which are also independent of $\xi_1$, $\xi_2,\ldots$. 
We claim that the so defined $X$ can equivalently be represented as follows:
\bel{eq:Levy_repr_sum2}
X(n)=X(0)+S_\xi(n-T(n-1))+  S_{\eta}(T(n-1)),\quad n\in\mn,
\ee
where $S_\xi(0)=S_\eta(0)=0$ a.s.

To check this, write $$X(n+1)-X(n)=S_\xi(n+1-T(n))-S_\xi(n-T(n-1))+S_{\eta}(T(n))-S_\eta(T(n-1)).$$ Observe now that $X(n)\neq 0$ if, and only if, $T(n-1)=T(n)$ and that on this event $$X(n+1)-X(n)=S_\xi(n+1-T(n))-S_\xi(n-T(n))=\xi_{n+1-T(n)},$$ which is in line with \eqref{tildeS}. On the other hand, $X(n)=0$ if, and only if, $T(n-1)=T(n)-1$ and on this event $$X(n+1)-X(n)=S_\eta(T(n))-S_\eta(T(n)-1)=\eta_{T(n)},$$ which is again in agreement with \eqref{tildeS}.

Put $T(-1)=0$. Using \eqref{eq:Levy_repr_sum2} with $X_v(n)$ replacing $X(n)$, $X_v(0)$ replacing $X(0)$ and $T_v(n)$ replacing $T(n)$, where $T_v$ is a counterpart of $T$ which corresponds to $X_v$, we conclude that  relation \eqref{eq:FCLT2} holds if we can show that
\begin{equation}\label{eq:rel1}
\Big(\frac{S_\xi(\lfloor vt\rfloor-{T_v}(\lfloor vt\rfloor-1))}{a(v)}\Big)_{t\geq 0}~\Rightarrow~ U_\alpha,\quad v\to\infty
\end{equation}
on $D$ and, for all $t_0>0$,
\begin{equation}\label{eq:rel2}
\frac{\sup_{t\in [0,\,t_0]}\,|S_\eta({T_v}(\lfloor vt\rfloor-1))|}{a(v)}~\overset{\mmp}{\to}~0,\quad v\to\infty.
\end{equation}

Assume that we can prove that,
for all $\delta>0$ and all $t\geq 0$,
\begin{equation}\label{eq:rel4}
\frac{{T_v(\lfloor vt\rfloor)} 
}
{{v}
^{1-1/\alpha+\delta}}~\overset{\mmp}{\to}~0,\quad {v} 
\to\infty.
\end{equation}
The sequence $(T_v(n))_{n\in\mn_0}$ is a.s.\ nondecreasing, and formula \eqref{eq:rel4} implies that, for all $t\geq 0$, $v^{-1}T_v(\lfloor vt\rfloor)\overset{\mmp}{\to}~0$ as $v\to\infty$. Hence, 
for all $t_1>0$, $$\sup_{t\in [0,\,t_1]}\,\Big|\frac{\lfloor vt\rfloor-{T_v}(\lfloor vt\rfloor-1)}{v}-t\Big|~\overset{\mmp}{\to}~0,\quad v\to\infty.$$  
Here, the limit function is deterministic, increasing and continuous. The latter limit relation can be combined with \eqref{eq:convStable1823} into $$\Big(\frac{S_\xi(\lfloor vt\rfloor)}{a(v)},\frac{\lfloor vt\rfloor-{T_v}(\lfloor vt\rfloor-1)}{v}\Big)_{t\geq 0}\Rightarrow~(U_\alpha, I),\quad v\to\infty$$ on $D$, where $I(t):=t$ for $t\geq 0$. The left-hand (right-hand) side in \eqref{eq:rel1} is composition of the coordinates on the left-hand (right-hand) side of the last limit relation. By Proposition \ref{prop:WhittConvCompositions} and Remark \ref{remk:corl_Skor_repr}, \eqref{eq:rel1} follows.

To prove \eqref{eq:rel2}, write, for any $\gamma>0$,
\begin{multline*}
\frac{\sup_{t\in [0,\,t_0]}\,|S_\eta({T_v}(\lfloor vt\rfloor-1))|}{a(v)}\leq
\frac{S_{|\eta|}({T_v}(\lfloor vt_0 \rfloor))}{a(v)}\\\leq \frac{S_{|\eta|}({T_v}(\lfloor vt_0 \rfloor))}{a(v)}\1_{\{{T_v}(\lfloor vt_0 \rfloor)>\gamma v^{1-1/\alpha+\delta} \}}+\frac{S_{|\eta|}(\lfloor \gamma v^{1-1/\alpha+\delta} \rfloor)}{a(v)}.
\end{multline*}
In view of \eqref{eq:rel4}, the first term on the right-hand side converges to $0$ in probability, as $v\to\infty$. To analyze the second term, recall that the function $a$ is regularly varying at $\infty$ of index $1/\alpha$.

\noindent If $\me |\eta|<\infty$, then choosing any $\delta\in (0, 2/\alpha-1)$ and invoking the weak law of large numbers for $S_{|\eta|}$ we infer
\begin{equation}\label{eq:int1}
\frac{S_{|\eta|}(\lfloor \gamma v^{1-1/\alpha+\delta} \rfloor)}{a(v)}~\overset{\mmp}{\to}~0,\quad v\to\infty.
\end{equation}
If the distribution of $\eta$ belongs to the domain of attraction of a $\beta$-stable
distribution with $\beta\in (\alpha-1, 1)$, then so does the distribution of $|\eta|$, and according to \eqref{eq:conv}, $S_{|\eta|}(v)/c(v)$ converges in distribution to a positive $\beta$-stable random variable. For any $\delta\in (0, \alpha^{-1}(\beta-(\alpha-1)))$ (such a choice is possible because $\beta>\alpha-1$) $\lim_{v\to\infty}c(\lfloor \gamma v^{1-1/\alpha+\delta} \rfloor)/a(v)=0$ and thereupon \eqref{eq:int1} holds true.

It remains to prove 
\eqref{eq:rel4}. Observe that, for $v,x>0$ and $n\in\mn_0$, $\mmp\{T_v(n)>x| X_v(0)\neq 0\}\leq \mmp\{T_v(n)>x|X_v(0)=0\}$. In view of this, we assume in what follows that $X_v(0)=0$ a.s.\ and write $T(n)$ for $T_v(n)$. Relation \eqref{eq:rel4} holds if we can show that
\begin{equation}\label{eq:rel454}
\frac{T(n)}{n^{1-1/\alpha+\delta}}~\overset{\mmp}{\to}~0,\quad n\to\infty.
\end{equation}

The random variable $T(n)$ has the same distribution as $$T^\prime(n):=1+\sum_{k\geq 1}\1_{\{(1+\tau^{(1)}_{-\eta_1})+\ldots+(1+\tau^{(k)}_{-\eta_k})\leq n\}},\quad n\in\mn_0,$$ where $(\tau^{(1)}_k)_{k\in\mathbb{Z}}$, $(\tau^{(2)}_k)_{k\in\mathbb{Z}},\ldots$ are independent copies of $(\tau_k)_{k\in\mathbb{Z}}$, which are also independent of $\eta_1$, $\eta_2,\ldots$. Fix any $n_0\in \mathbb{Z}\backslash\{0\}$ satisfying $$p_0:=\mmp\{X(1) = n_0 \ |\ X(0)=0\}=\mmp\{\eta=n_0\}>0.$$ Put $\theta_0:=0$ and $\theta_{i+1}:=\inf\{j>\theta_i: \eta_j=n_0\}$ for $i\in\mn_0$. The random variables $\theta_1$, $\theta_2-\theta_1,\ldots$ are independent and have a geometric distribution with success probability $p_0$, that is, $\mmp\{\theta_1=k\}=(1-p_0)^{k-1}p_0$ for $k\in\mn$. Also, $\theta_1$, $\theta_2,\ldots$ are independent of $(\tau^{(1)}_k)_{k\in\mathbb{Z}}$, $(\tau^{(2)}_k)_{k\in\mathbb{Z}},\ldots$ Write
\begin{multline*}
T^\prime(n)-1=\sum_{i\geq 0}\sum_{k=\theta_i+1}^{\theta_{i+1}}\1_{\{(1+\tau^{(1)}_{-\eta_1})+\ldots+(1+\tau^{(k)}_{-\eta_k})\leq n\}}\\\leq \theta_1+\sum_{i\geq 1}(\theta_{i+1}-\theta_i)\1_{\{(1+\tau^{(\theta_1)}_{-\eta_{\theta_1}})+\ldots+(1+\tau^{(\theta_i)}_{-\eta_{\theta_i}})\leq n\}}\\=\theta_1+\sum_{i\geq 1}(\theta_{i+1}-\theta_i)\1_{\{(1+\tau^{(\theta_1)}_{-n_0})+\ldots+(1+\tau^{(\theta_i)}_{-n_0})\leq n\}}\quad\text{a.s.}
\end{multline*}
The latter random variable has the same distribution as $$\theta_1+\sum_{i\geq 1}(\theta_{i+1}-\theta_i)\1_{\{(1+\tau^{({1} 
)}_{-n_0})+\ldots+(1+\tau^{({i} 
)}_{-n_0})\leq n\}}=\sum_{k=1}^{T^\ast(n)}(\theta_k-\theta_{k-1}),$$ where $T^\ast(n):=1+\sum_{i\geq 1}\1_{\{(1+\tau^{(1)}_{-n_0})+\ldots+(1+\tau^{(i)}_{-n_0})\leq n\}}$, for $n\in\mn_0$, is independent of $\eta_1$, $\eta_2,\ldots$. Summarizing, to prove {\eqref{eq:rel454}} 
it is enough to show that, for all $\delta>0$, $$n^{-(1-1/\alpha+\delta)}\sum_{k=1}^{T^\ast(n)}(\theta_k-\theta_{k-1})~\overset{\mmp}{\to}~0,\quad n\to\infty.$$ By the weak law of large numbers for the random walk $(\theta_i)_{i\in\mn}$, the latter holds provided that $$\frac{T^\ast(n)}{n^{1-1/\alpha+\delta}}~\overset{\mmp}{\to}~0,\quad n\to\infty.$$ According to Lemma 2.1 in \cite{Belkin:1970}, $\mmp\{\tau_0>n\}~\sim~ n^{-(1-1/\alpha)}L_1(n)$ as $n\to\infty$ for some $L_1$ slowly varying at $\infty$. By Theorem T1 on p.~378 in \cite{Spitzer:2001},
\begin{equation}\label{eq:ratio}
\lim_{n\to\infty} \frac{\mmp\{\tau_{-n_0}>n\}}{\mmp\{\tau_0>n\}}=g(n_0)\in [0,\infty),
\end{equation}
where $g$ is the potential kernel of $S_\xi$. By Theorem P2 on p.~361 of the same reference, $g(k)>0$ for all $k\in\mathbb{Z}\backslash\{0\}$ and particularly $g(n_0)>0$. This follows from the fact that $\xi$ has a symmetric distribution with unbounded support, so that $S_\xi$ cannot be a left- or right-continuous random walk.

The sequence $(T^\ast(n))_{n\in\mn_0}$ is the first-passage time (generalized inverse) sequence for $S_{1+\tau_{-n_0}}$. In view of \eqref{eq:ratio} and $g(n_0)>0$, the distribution tail of $1+\tau_{-n_0}$ is regularly varying at $\infty$ of index $-(1-1/\alpha)\in (-1,0)$. Then $\mmp\{\tau_{-n_0}>n\}T^\ast(n)$ converges in distribution to a random variable having a Mittag-Leffler distribution (the distribution of an inverse $(1-1/\alpha)$- subordinator evaluated at time $1$), see, for instance, Theorem 7 in \cite{Feller:1949}. Since, for all $\delta>0$ and any $L^\ast$ slowly varying at $\infty$, $\lim_{n\to\infty}n^\delta L^\ast(n)=\infty$, relation {\eqref{eq:rel454}} 
follows.

The proof of Theorem \ref{thm:lim_skew_Levy_walk}(b) is complete.

\section{Appendix}

\begin{proof}[Proof of Proposition \ref{thm:weak_convergenceEK_chains}]
If condition \eqref{eq:conv_semigroups_restr} prevails, the result follows from a specialization of Theorem 2.11 on p.~172 in \cite{Ethier+Kurtz:2005}.

If condition \eqref{eq:conv_resolvents_restr} holds, then we argue along the lines of the proof of Trotter's approximation theorem (Theorem 4.2 on p. 85 in \cite{Pazy:1983}). An additional useful information can be found in Section 3.6 of the cited book.

We intend to show that \eqref{eq:conv_resolvents_restr} entails \eqref{eq:conv_semigroups_restr}. Let $(T_A(t))_{t\geq 0}$ and $(T_B(t))_{t\geq 0}$ be strongly continuous semigroups defined on Banach spaces $\mathcal A$ and $\mathcal B$ with infinitesimal generators $A$ and $B$ and resolvents $R_\lambda(A):=(\lambda I-A)^{-1},\ {\lambda>0}$ and $R_\lambda(B):=(\lambda I-B)^{-1}$, $\lambda>0$, respectively. Let $\pi: \mathcal A \to \mathcal B$ be
a continuous linear operator. We claim that, for $f\in \mathcal A$,
\bel{eq:resolvents_eq0}
R_\lambda(B)\Big(\pi T_A(t) - T_B(t) \pi  \Big)   R_\lambda(A) f=
\int_0^t  T_B(t-s) \Big(\pi R_\lambda(A)  -  R_\lambda(B)  \pi  \Big) T_A(s) f {\rm d}s,\quad t\geq 0.
\ee
Indeed,
\begin{multline}\label{eq:resolvents_eq1}
\frac{{\rm d}}{{\rm d}s}\big[T_B(t-s)R_\lambda(B)\pi T_A(s)  R_\lambda(A)\big]f\\=
-T_B(t-s) B R_\lambda(B)\pi T_A(s)  R_\lambda(A) f+
T_B(t-s)   R_\lambda(B)\pi T_A(s) A R_\lambda(A) f\\=
-T_B(t-s) B R_\lambda(B)\pi R_\lambda(A)  T_A(s)  f+
T_B(t-s)   R_\lambda(B)\pi  A R_\lambda(A) T_A(s) f\\=
T_B(t-s)\left[ - B R_\lambda(B)\pi R_\lambda(A)  +
R_\lambda(B)\pi  A R_\lambda(A)\right] T_A(s) f.
\end{multline}
Recall that $AR_\lambda(A)= \lambda R_\lambda(A)-I.$ Hence, the right-hand side of \eqref{eq:resolvents_eq1} is equal to
\begin{multline*} 
 T_B(t-s)\left[ - (\lambda R_\lambda(B)-I)\pi R_\lambda(A)  +
    R_\lambda(B)\pi  (\lambda R_\lambda(A)-I)\right] T_A(s) f=\\
    T_B(t-s)\left[ \pi R_\lambda(A)  -  R_\lambda(B)  \pi\right] T_A(s) f,
\end{multline*}
and integration in $s\in [0,t]$ yields \eqref{eq:resolvents_eq0}.

For $n\in\mn_0$, denote by $A^{(n)}$ the infinitesimal generator of $X^{(n)}$. For $n\in\mn$, denote by $B(G_n)$ the Banach space of bounded measurable functions on $G_n$ with the supremum norm $\|\cdot\|$ (the same notation as for the supremum norm in $C_0(\mbR)$) and by $\pi_n f$ the restriction of $f\in C_0(\mbR)$ to $G_n$. Getting back to the setting of the Markov processes, put $\mathcal A:= C_0(\mbR)$, $A:=A^{(0)}$, $T_A(t):=P^{(0)}(t)$ and, for each $n\in\mn$, $\mathcal B:= B(G_n)$, $B:=A^{(n)}$, $T_B(t):=P^{(n)}(t)$ and $\pi f:=\pi_n f$.

For $f\in C_0(\mbR)$,
\begin{multline}\label{eq:1556_res}
\|(P^{(n)}(t) \pi_n -\pi_n P^{(0)}(t)) R_\lambda(A^{(0)}) f\|\leq
\\
\|P^{(n)}(t) (\pi_n R_\lambda(A^{(0)}) - R_\lambda(A^{(n)})\pi_n)  f\|+
\|R_\lambda(A^{(n)})(P^{(n)}(t) \pi_n  - \pi_n P^{(0)}(t))  f\|+\\
\|(R_\lambda(A^{(n)})\pi_n-\pi_n R_{\lambda}(A^{(0)}))P^{(0)}(t) f\|.
\end{multline}
Relation \eqref{eq:conv_resolvents_restr} ensures that the first and the third terms on the right-hand
side of \eqref{eq:1556_res} converge to $0$ as $n\to\infty$. When analyzing the second term on the right-hand side of \eqref{eq:1556_res} we first assume that
$f= R_\lambda(A^{(0)})g$ for $g\in C_0(\mbR)$. It then follows from \eqref{eq:resolvents_eq0} that
\begin{multline*}
\|R_\lambda(A^{(n)})(P^{(n)}(t) \pi_n  - \pi_n P^{(0)}(t))  R_\lambda(A^{(0)})g\|=\\
\Big\|\int_0^t P^{(n)}(t)\big( \pi_n R_\lambda(A^{(0)})- R_\lambda(A^{(n)})\pi_n P^{(0)}(t) \big) P^{(0)}(s)g {\rm d}s\Big\|.
\end{multline*}
The right-hand side converges to $0$ as $n\to\infty$ by \eqref{eq:conv_resolvents_restr} and the Lebesgue dominated convergence theorem. Thus, we have shown that
\bel{eq:EK_conv_semigr_proj}
\|(P^{(n)}(t)\pi_n-\pi_n P^{(0)}(t)) h\|\to 0,\quad n\to\infty
\ee
with $h=R_\lambda(A^{(0)}) f= (R_\lambda(A^{(0)}))^2 g$ where $g$ is an arbitrary function from $C_0(\mbR)$, that is, \eqref{eq:EK_conv_semigr_proj} holds for any $h\in {\rm Dom}\,((A^{(0)})^2)$, the domain of $(A^{(0)})^2$. By Theorem 2.7 on p.6 in \cite{Pazy:1983}, ${\rm Dom}\,((A^{(0)})^2)$ is dense in $C_0(\mbR)$. Hence, relation \eqref{eq:EK_conv_semigr_proj} holds for any $h\in C_0(\mr)$, which is equivalent to \eqref{eq:conv_semigroups_restr}.
\end{proof}
\bigskip
\noindent {\sc Acknowledgement}. The research was supported by the High Level Talent Project DL2022174005L
of Ministry of Science and Technology of PRC.

\end{document}